%% file: RPfeb25.tex
\documentclass{amsart}
\usepackage{color,epsf}
\usepackage{enumerate}
\usepackage{mathtools} 
\mathtoolsset{showonlyrefs,showmanualtags} 

\newcommand{\be}{\begin{equation}}
\newcommand{\ee}{\end{equation}}

\ifx\pdfpageheight\undefined
   \usepackage[dvips,colorlinks=true,linkcolor=blue,citecolor=red,%
      urlcolor=green]{hyperref}
   \usepackage[dvips]{graphicx}
   \makeatletter
   \edef\Gin@extensions{\Gin@extensions,.mps}
   \makeatother
\else
 \usepackage[pdftex]{graphicx}
   \usepackage[bookmarksopen=false,pdftex=true,breaklinks=true,%
      backref=page,pagebackref=true,plainpages=false,%
      hyperindex=true,pdfstartview=FitH,colorlinks=true,%
      pdfpagelabels=true,colorlinks=true,linkcolor=blue,%
      citecolor=red,urlcolor=green,hypertexnames=false%
      ]%
   {hyperref}
\fi

\input{preamble}

\begin{document}

\title[Zeroes of polynomials on definable hypersurfaces]{Zeroes of polynomials on definable hypersurfaces: pathologies exist, but they are rare}

\author{Saugata Basu}
\address{Department of Mathematics, Purdue University, West Lafayette, IN 47906, U.S.A.}
\email{sbasu@math.purdue.edu}

\author{Antonio Lerario}
\address{SISSA (Trieste) and Florida Atlantic University}
\email{lerario@sissa.it}

\author{Abhiram Natarajan}
\address{Department of Computer Science, Purdue University, West Lafayette, IN 47906, U.S.A.}
\email{nataraj2@purdue.edu}

\date{\textbf{\today}}



\thanks{Basu was partially supported by NSF grants CCF-1618918 and DMS-1620271.}

\begin{abstract}Given a sequence $\{Z_d\}_{d\in \mathbb{N}}$ of smooth and compact hypersurfaces in $\R^{n-1}$, we prove that (up to extracting subsequences) there exists a regular definable hypersurface $\Gamma\subset \RP^n$ such that each manifold $Z_d$ appears as a component of the zero set on $\Gamma$ of some polynomial of degree $d$. (This is in sharp contrast with the case when $\Gamma$ is algebraic, where for example the homological complexity of the zero set of a polynomial $p$ on $\Gamma$ is bounded by a polynomial in $\deg(p)$.) 

More precisely, given the above sequence of hypersurfaces, we construct a regular, compact, definable hypersurface $\Gamma\subset \RP^{n}$ containing a subset $D$ homeomorphic to a disk, and a family of polynomials $\{p_m\}_{m\in \mathbb{N}}$ of degree $\deg(p_m)=d_m$ such that $(Z(p_m)\cap D, D)\sim (Z_{d_m}, \R^{n-1}),$ i.e. the zero set of $p_m$ in $D$ is isotopic to $Z_{d_m}$ in $\R^{n-1}$.

This says that, up to extracting subsequences, the intersection of $\Gamma$ with a hypersurface of degree $d$ can be as complicated as we want. We call these ``pathological examples''.

In particular, we show that for every $0 \leq k \leq n-2$ and every sequence of natural numbers $a=\{a_d\}_{d\in \mathbb{N}}$ there is a regular, compact and definable hypersurface $\Gamma\subset \RP^n$,  a subsequence $\{a_{d_m}\}_{m\in \mathbb{N}}$ and homogeneous polynomials $\{p_{m}\}_{m\in \mathbb{N}}$ of degree $\deg(p_m)=d_m$ such that:
\be \label{eq:pathintro} b_k(\Gamma\cap Z(p_m))\geq a_{d_m}.\ee
(Here $b_k$ denotes the $k$-th Betti number.) This generalizes a result of Gwo\'zdziewicz, Kurdyka and Parusi\'nski \cite{gwozdziewicz1999number}.

On the other hand, for a given definable $\Gamma$ we show that the Fubini-Study measure, in the gaussian space of polynomials of degree $d$, of the set $\Sigma_{d_m,a,  \Gamma}$ of polynomials verifying \eqref{eq:pathintro} is positive, but there exists a contant $c_\Gamma$ such that this measure can be bounded by:
\be 
0<\PP(\Sigma_{d_m, a, \Gamma})\leq \frac{c_{\Gamma} d_m^{\frac{n-1}{2}}}{a_{d_m}}.
\ee
This shows that the set of ``pathological examples'' has ``small'' measure (the faster $a$ grows, the smaller the measure and pathologies are therefore rare). In fact we show that given $\Gamma$, for most polynomials a B\'ezout-type bound holds for the intersection $\Gamma\cap Z(p)$: for every $0\leq k\leq n-2$ and $t>0$:
\be
\label{eq:curve}\PP\left(\{b_k(\Gamma\cap Z(p))\geq t d^{n-1} \}\right)\leq \frac{c_\Gamma}{td^{\frac{n-1}{2}}}.
\ee
\end{abstract}

\maketitle


\input{introduction}

\input{pathological-examples-construction}
\input{average-betti-numbers}

\bibliographystyle{amsplain}

\bibliography{rp}

\end{document}

%% file: preamble.tex
\usepackage{amsmath,amssymb,amsfonts}
\newtheorem{theorem}{Theorem}

\newtheorem{corollary}{Corollary}
\newtheorem{proposition}{Proposition}

\theoremstyle{definition}

\theoremstyle{remark}
\newtheorem{remark}{Remark}

\definecolor{DarkBlue}{rgb}{0,0.1,0.55}

\numberwithin{equation}{section}

\newcommand {\hide}[1]{}

\newcommand {\junk}[1]{}

\newcommand {\R} {\mathbb{R}}

\newcommand {\PP}     {\mathbb{P}} 

\def\addots{\mathinner{\mkern1mu
\raise1pt\vbox{\kern7pt\hbox{.}}
\mkern2mu\raise4pt\hbox{.}\mkern2mu
\raise7pt\hbox{.}\mkern1mu}}

\newcommand {\RP}   {\mathbb{R}\mathrm{P}}


%% file: introduction.tex
\section{Introduction}

\subsection{Existence of pathologies}A classical fact from algebraic geometry states that given two real algebraic curves $\Gamma$ and $Z$, if their intersection is finite, it consists of at most $\deg(\Gamma)\cdot\deg(Z)$ many points (this is a consequence of B\'ezout's Theorem). In particular, if we fix the first curve, we can say that there is a function $\beta_{\Gamma,0}:\mathbb{N}\to \mathbb{N}$ such that for every polynomial $p$ of degree $d$, if $\Gamma$ and $Z(p)=\{p=0\}$ intersect at finitely many points, then:
\be\label{eq:basic} \#(\Gamma\cap Z(p))\leq \beta_{\Gamma, 0}(d)=\deg(\Gamma) \cdot d.\ee
If we leave the \emph{semialgebraic} world, but still remain in the \emph{definable} setting, still such a function $\beta_{\Gamma, 0}$ exists, but in general nothing can be said about its behavior. Here by definable we mean the class of definable sets in an o-minimal expansion of the real numbers, for example the o-minimal structure generated by subanalytic functions. (We refer the reader
who is unfamiliar with o-minimal geometry to \cite{Dries,Michel2} for easy to read introductions to the topic.)

In this direction Gwo\'zdziewicz, Kurdyka and Parusi\'nski \cite{gwozdziewicz1999number} have proved that for every sequence $\{a_d\geq0\}_{d\in \mathbb{N}}$ of natural numbers there exists a definable curve $\Gamma$, a subsequence $\{a_{d_m}\}_{m\in \mathbb{N}}$ and a sequence $\{p_m\}_{m\in \mathbb{N}}$ of polynomials of degree $\deg(p_m)=d_m$ such that:
\be  \#(\Gamma\cap Z(p_m))\geq a_{d_m}.\ee
(In this paper we will show that the curve $\Gamma\subset \RP^2$ can be taken to be regular, definable and compact and that the polynomials $p_{d_m}$ can be chosen in such a way that the intersection $\Gamma\cap Z(p_{d_m})$ is transversal, i.e. stable under small perturbations of the polynomial.)

In particular this shows that, for a fixed definable $\Gamma\subset \RP^2$, there is in general no upper bound on the number of zeroes of a polynomial $p$ on $\Gamma$ which is polynomial in $\deg(p)$. Generalizing this we will show that in higher dimensions the situation is even more interesting. 

To state our first result, we will say that two manifold pairs $(X, M)$ and $(Y, N)$ are diffeomorphic if there exists a diffeomorphism $\psi:M\to N$ such that $\psi(X)=Y$; in this case we write $(X, M)\sim (Y,N)$. This notion essentially says that $X$ and $Y$ are diffeomorphic and, up to a diffeomorphim, they are embedded in their ambient spaces in the same way. 

Of course, when $\Gamma$ is an algebraic hypersurface and $p$ is a polynomial, there are restrictions on the possible pairs $(Z(p)\cap \Gamma, \Gamma)$ (for example Betti numbers of $Z(p)\cap \Gamma$ grow at most as a polynomial in $\deg(p)$). Pick now a sequence of smooth and compact hypersurfaces $Z_1, Z_2, \ldots \subset \R^{n-1}$. Our first Theorem says that (up to extracting subsequences) there exists a regular definable hypersurface $\Gamma\subset \RP^n$ such that each manifold $Z_d$ appears as a component of the zero set on $\Gamma$ of some polynomial of degree $d$. More precisely, we will prove the following.
\begin{theorem}[Existence of pathologies]
\label{thm:arbitrary-intersection}
Let $\{Z_d\}_{d\in \mathbb{N}}$ be a sequence of smooth, compact hypersufaces embedded in $\R^{n-1}$. There exist a regular\footnote{Throughout the paper the word ``regular'' will mean ``of regularity class $C^k$ for some fixed $k\geq 2$''.}, compact, definable hypersurface $\Gamma\subset \RP^n$, a disk $D\subset \Gamma$ and a sequence $\{p_m\}_{m\in \mathbb{N}}$ of homogeneous polynomials of degree $\deg(p_m)=d_m$ such that the intersection $Z(p_m)\cap \Gamma$ is transversal and:
\be (Z(p_m)\cap D, D)\sim (Z_{d_m}, \R^{n-1})\quad \textrm{for all $m\in \mathbb{N}$}.\ee
\end{theorem}
\begin{remark}Note that in the case $n=2$ this implies the statement of \cite{gwozdziewicz1999number}. In fact, we can take for $Z_d=\{x_1, \ldots, x_{a_d}\} \subset \R$ a set consisting of $a_d$ many points. Then we find a smooth definable curve $\Gamma\subset \RP^2$, an interval $I\subset \Gamma$ and a sequence of polynomials $p_m$ fo degree $d_m$ such that the manifold pairs $(Z(p_m)\cap I, I)$ and $(\{x_1, \ldots, x_{a_{d_m} }\}, \R)$ are diffeomorphic, in particular $Z(p_m)\cap \Gamma$ consists of at least $a_{d_m}$ many points.
\end{remark}
In higher dimensions we can measure the complexity of a manifold by its Betti numbers. If $\Gamma\subset \RP^n$ is a regular, compact, definable hypersurface, for every $0\leq k\leq n-2$ let $\beta_{\Gamma, k}:\mathbb{N}\to\mathbb{N}$ be the function:
\be \beta_{\Gamma, k}(d)=\max_{\deg(p)=d}b_k(\Gamma\cap Z(p))\ee
(here $b_k$ denotes the $k$-th Betti number, and the maximum is taken over nondegenerate intersections). When $\Gamma$ is algebraic, we have 
\be\label{eq:bettialg}\beta_{\Gamma, k}(d)\leq c_\Gamma\cdot d^{n-1}\quad \textrm{(algebraic case)}\ee
for some constant depending on the degree of $\Gamma$ (this estimate actually requires some nontrivial work if $\Gamma$ is singular, and it is proved in \cite[Theorem 6.4]{basu2015multi}). On the other hand, as for the case of curves, there is no way to control the behavior of this function for a general definable $\Gamma:$ in fact, given a sequence $\{a_d\}_{d\in \mathbb{N}}$, if we chose a sequence of hypersurfaces $\{Z_d\}$ with $b_k(Z_d)\geq a_d$, for the hypersurface $\Gamma$ provided by Theorem \ref{thm:arbitrary-intersection} the function $\beta_{\Gamma, k}$ grows at least as fast as $a_{d_m}$. 

\begin{remark}\label{rem:partition}Estimates like \eqref{eq:basic} are basic building blocks in recent advances in incidence
problems in the area of discrete geometry driven by the polynomial partitioning
method \cite{Guth-Katz} (see for example \cite[Theorem A.2]{Solymosi-Tao}). Recently, using different techniques
such incidence results have
been generalized from the semi-algebraic case to more general situations --
namely, incidences between definable sets over arbitrary o-minimal expansions of
$\R$, see \cite{Basu-Raz, Galvin-Chernikov-Starchenko}. 
In order to
extend the polynomial partitioning technique to the o-minimal situation
(as noted in \cite{OWF-report})
it is important to study the function $\beta_{\Gamma,k}$
where $\Gamma$ is now an arbitrary definable hypersurface in an o-minimal structure
(rather than just semi-algebraic).
On one hand Theorem \ref{thm:arbitrary-intersection} seems to rule out the use of polynomial partitioning for incidence problems involving definable
sets in arbitrary o-minimal structures, but on the other hand we also prove (see Theorem \ref{thm:estimate} below)  
that the pathological behavior  exhibited in Theorem 
\ref{thm:arbitrary-intersection} is very rare, and this gives some hope that a modified version of the technique can still be applicable to incidence questions.  
\end{remark}

\subsection{Pathologies are rare}

Given $\Gamma$, it is natural to ask how ``stable'' are the polynomials having the ``pathological'' behaviour of Theorem \ref{thm:arbitrary-intersection}? In other words, if it is certainly true that nothing can be said on the function $\beta_{\Gamma, k}$ that bounds the Betti numbers of transversal intersection between a definable hypersurface $\Gamma$ and the zero set of a polynomial in terms of the degree of the polynomial, is it possible to say that for ``most polynomials'' a polynomial upper bound still holds true for the Betti numbers? Our second result gives an affirmative answer to this question, after the naive idea of ``most polynomials'' is made precise.

To make these questions precise, on the space $W_{n,d}$ of homogeneous polynomials of degree $d$ in $n+1$ variables we introduce a natural gaussian measure, called the \emph{Kostlan} measure, defined by choosing each coefficient of 
\be p=\sum_{|\alpha|=d}\xi_\alpha {d \choose \alpha}^{1/2}x_0^{\alpha_0}\cdots x_n^{\alpha_n}\ee
 independently from a standard Gaussian distribution (i.e. $\xi_\alpha\sim N(0,1)$). This measure is the restriction, to the space of real polynomials, of the Fubini-Study measure.
 
The scaling coefficients ${d\choose \alpha}^{1/2}$ are chosen in such a way that the resulting probability distribution is invariant under orthogonal change of variables (there are no preferred points or direction in $\RP^n$, where zeroes of $p$ are naturally defined). Moreover, if we extend this probability distribution to the whole space of complex polynomials, by replacing \emph{real} with \emph{complex} gaussian variables, it can be shown that this extension is the \emph{unique} Gaussian measure which is invariant under unitary change of variables. This makes real Kostlan polynomials a natural object of study. (This model for random polynomials received a lot of attention since pioneer works of Edelman, Kostlan, Shub and Smale \cite{EdelmanKostlan95, shsm, EKS, Ko2000, ShSm3, ShSm1} on random polynomial systems solving.) 

The next Theorem estimates the size of the set of polynomials whose restriction to a definable hypersurface $\Gamma\subset \RP^n$ have a behaviour that deviates from the algebraic case estimate \eqref{eq:bettialg}.

\begin{theorem}\label{thm:estimate}Let $\Gamma\subset \mathbb{R}\emph{P}^n$ be a regular and compact hypersurface, and let $p$ be a random Kostlan polynomial of degree $d$. Then there exists a constant $c_\Gamma$ such that for every $0\leq k\leq n-2$ and  for every $t>0$
\be\label{eq:curve}\PP\{b_k(\Gamma\cap Z(p))\geq t d^{n-1} \}\leq \frac{c_\Gamma}{td^{\frac{n-1}{2}}}.\ee
\end{theorem}
Combining this result with the construction of Theorem \ref{thm:arbitrary-intersection} we obtain the following estimate for the gaussian volume of the set of ``pathological'' polynomials. The lower bounds follows from the fact that the intersection $Z(p_m)\cap \Gamma$ produced in Theorem \ref{thm:arbitrary-intersection} is transversal (hence stable under small perturbations of the polynomial $p_m$). 
\begin{corollary}[Pathologies are rare]Given a sequence of natural numbers $\{a_d\}_{d\in \mathbb{N}}$ let $\{Z_{d}\}_{d\in \mathbb{N}}$ be a sequence of hypersurfaces with $b_k(Z_d)\geq a_d$ for all $d\in \mathbb{N}$. Consider the hypersurface $\Gamma\subset \mathbb{R}\textrm{\emph{P}}^n$ provided by Theorem \ref{thm:arbitrary-intersection}. Then, for some constant $c_\Gamma>0$:
\be 0<\PP\{b_k(\Gamma\cap Z(p))\geq a_{d_m}\}\leq \frac{c_\Gamma d_m^{\frac{n-1}{2}}}{a_{d_m}}. \ee
\end{corollary}
\begin{remark}The conclusion of Theorem \ref{thm:estimate} follows after combining Markov's inequality with the following fact (proved in Proposition \ref{prop:bettimean}): there exists a universal constant $c_{k,n}>0$ such that for every $\Gamma\subset \RP^n$ regular, definable, compact hypersurface
\be \label{eq:bettimean}\mathbb{E}b_k(\Gamma\cap Z(p))\leq |\Gamma| c_{k, n}d^{\frac{n-1}{2}}+O(d^{\frac{n-2}{2}}),\ee
where $|\Gamma|$ denotes the volume of $\Gamma$, induced by restricting the riemannian metric of $\RP^n$, and the implied constants in the $O(d^{\frac{n-2}{2}})$ depends on $\Gamma$. 
This result could be potentially useful in the study of incidence questions over o-minimal structures (cf. Remark \ref{rem:partition}).
\end{remark}
\begin{remark}The content of \eqref{eq:bettimean} reveals an interesting and surprising property of the space of polynomials: by Theorem \ref{thm:arbitrary-intersection} there is a priori no upper bound on the homological complexity of $\Gamma\cap Z(p)$ (as a function of $d=\mathrm{deg}(p)$), but \emph{on average} we cannot exceed a polynomial bound. Here is an example from \cite{flats} of a similar phenomenon that appears in the study of random enumerative geometry. If $X_1,\ldots, X_4$ are boundaries of smooth convex bodies in $\RP^3$, one can ask for the number $\ell(X_1, \ldots, X_4)$ of lines that are simultaneously tangent to all of them. This number is finite if the convex bodies are in general position in the projective space, but it can be arbitrarily large: for every $m>0$ one can find $X_1, \ldots, X_4\subset \RP^3$ in general position such that there are at least $m$ lines tangent to all of them. On the other hand (here is the surprising thing) there exists a constant $c>0$, independent of the convex bodies, such that if we now average over all their possible configurations using the action of the orthogonal group $O(4)$ on $\RP^3$, we get $\mathbb{E}_{g_1 ,\ldots, g_4\in O(4)} \ell(g_1X_1, \ldots, g_4X_4)=c.$ Here again there is no a priori upper bound, but there is an upper bound \emph{on average}.
\end{remark}
\begin{remark}[The zero-dimensional case]Another case of interest, on which we can say more, is the case when $\Gamma\subset \RP^n$ is $k$-dimensional and we consider the common zero set of $k$ polynomials on it. In this case we do not have to restrict to Kostlan polynomials and we can work with the more general class of random \emph{invariant} polynomials: these are centered Gaussian probability measure on $W_{n,d}$ which are invariant under the action of the orthogonal group by change of variables (of course the Kostlan measure is one of them). These measures have been classified by Kostlan \cite{Kostlan} and depend on $\lfloor \frac{d}{2}\rfloor$ many parameters. Consider now the common zero set $X$ of independent random \emph{invariant} polynomials $p_1, \ldots, p_k$ on $\Gamma$:
\be X= \Gamma\cap Z(p_1)\cap\cdots \cap Z(p_k).\ee
With probability one $X$ is zero-dimensional and we can use integral geometry (see \cite{howard} or the appendix of \cite{PSC}) to deduce that:
\be\label{eq:ig}
\mathbb{E}\#\left(\Gamma\cap Z(p_1)\cap\cdots \cap Z(p_k)\right)=\frac{|\Gamma|}{|\RP^k|}\prod_{j=1}^k\mathbb{E}\frac{|Z(p_j)|}{|\RP^{n-1}|}.\ee
The quantity $\mathbb{E}|Z(p)|$ appearing in \eqref{eq:ig} can be evaluated using the definition of the invariant distribution in terms of its weights (see \cite{Kostlan,FLL}); when $p$ is a Kostlan polynomial of degree $d$, then $\mathbb{E}|Z(p)|=\sqrt{d}|\RP^{n-1}|$. More generally (again by Integral Geometry) this expectation is bounded by $\mathbb{E}|Z(p)|\leq d|\RP^{n-1}|$. If each $p_i$ has now degree $d$, we can apply Markov again and deduce that there exists $c_{\Gamma}>0$ such that for any  invariant gaussian measure on the space of polynomials:
\be \mathbb{P}\{\#(\Gamma\cap Z(p_1)\cap\cdots \cap Z(p_k))\geq td^{n-1}\}\leq \frac{c_\Gamma}{t}.\ee
(The probability of deviating from a B\'ezout-type bound is small.)
\end{remark}

%% file: pathological-examples-construction.tex
\section{Pathological examples: Proof of Theorem \ref{thm:arbitrary-intersection}}
\subsection{Some basic facts}For the next proof we will need a few elementary facts from differential topology and real algebraic geometry. First, if $D\subset\R^{n-1}$ is a disk and $f:\overline{D}\to\R$ is a regular function, we define:
\be \|f\|_{C^1(D, \R)}=\sup_{z\in D}\|f(z)\|+\sup_{z\in D}\|\nabla f(z)\|.\ee
If ``zero'' is a regular value of $f$, then $Z(f)$ is a regular hypersurface in $D$. If $Z \subset \R^{n-1}$ is a regular compact hypersuface we will write
\be (Z(f), D)\sim (Z, \R^{n-1})\ee
to denote that the two pairs $(Z(f), D)$ and $(Z, \R^{n-1})$ are diffeomorphic. In this setting there exists $\delta>0$ (depending on $f$) such that given any regular function $h:\overline{D}\to \R$ with $\|h\|_{C^1(D, \R)}\leq \delta$, ``zero'' is a regular value of $f+h$ and:
\be (Z(f+h), D)\sim (Z, \R^{n-1})\ee
(in particular the zero sets of $f$ and $h$ are diffeomorphic). We will (loosely) refer to this fact as \emph{Thom's isotopy Lemma}.

We will also need the followng classical approximation result from real algebraic geometry, due to Seifert \cite{Seifert}. Given a regular, compact hypersurface $Z\subset D\subset \R^{n-1}$, there exists a polynomial $q:\R^{n-1}\to \R$ such that ``zero'' is a regular value of $q$ and
\be (Z(q), D)\sim (Z, \R^{n-1}).\ee
This follows from Weirstrass' approximation Theorem; the reader can see \cite[Special case 5]{kollar} for an elementary proof of Seifert's result.

\begin{proof}[Proof of Theorem \ref{thm:arbitrary-intersection}] Let $e_1=(1, 0, \ldots, 0)\in \R^{n-1}$ and consider the two disks $D_1=D(e_1, \frac{1}{2})$ and $D_2=D(e_1, \frac{2}{3}).$

Pick $Z_1$ and consider a polynomial\footnote{We start with $q_2$ and not $q_1$, but the shift of the indices will be convenient to simplify the notation later.} $q_2$ such that:
\be (Z(q_2)\cap D_1, D_1)\sim (Z_1, \R^{n-1}).\ee
Observe that, since $\|x\|^2$ does not vanish on $D_1$,  ``zero'' is also a regular value for $Q_2=c_2\|x\|^2q_2|_{D_1}$ for every positive constant $c_2>0$, and:
\be (Z(Q_2)\cap D_1, D_1)\sim (Z_1, \R^{n-1}).\ee
(In the course of the proof we will pick a sequence of constants $\{c_k>0\}_{k\in \mathbb{N}}$ that will only be specified later.) 
Call $d_2$ the degree of $Q_2$ and observe that $Q_2$ only contains monomials $x_1^{\alpha_1}\cdots x_{n-1}^{\alpha_{n-1}}$ with $2\leq |\alpha|\leq d_2.$ (We set $d_1=1$.)

By Thom's isotopy Lemma, associated to the function $Q_2:D_1\to \R$ there is a $\delta_2>0$ such that for any other continuously differentiable function $h:D_1\to \R$ with $\|h\|_{C^1(D, \R)}\leq \delta_2$ we have that the equation $Q_2+h=0$ is regular on $D_1$ and the pair $(Z(Q_2+h)\cap D_1, D_1)$ is isotopic to the pair $(Z(Q_2)\cap D_1, D_1).$

Let now $k\geq 2$ and consider $Z_{d_k}$. Pick a polynomial $q_{k+1}$ such that ``zero'' is a regular value for $q_{k+1}|_{D_1}$ and:
\be (Z(q_{k+1})\cap D_1, D_1)\sim (Z_{d_k}, \R^{n-1}).\ee
As before, observe that ``zero'' is also a regular value for $Q_{k+1}=c_{k+1}\|x\|^{2d_k}q_{k+1}|_{D_1}$, for any positive constant $c_{k+1}>0$ and:
\be (Z(Q_{k+1})\cap D_1, D_1)\sim (Z_{d_k}, \R^{n-1}).\ee

Again, as before by Thom's isotopy Lemma, associated to the function $Q_{k+1}:D_1\to \R$ there is a $\delta_{k+1}>0$ such that for any other continuously differentiable function $h:D_1\to \R$ with $\|h\|_{C^1(D, \R)}\leq \delta_{k+1}$ we have that the equation $Q_{k+1}+h=0$ is regular on $D_1$ and the pair $(Z(Q_{k+1}+h)\cap D_1, D_1)$ is isotopic to the pair $(Z(Q_{k+1})\cap D_1, D_1).$ Moreover, calling $d_{k+1}=\deg(Q_{k+1})$, we have that $Q_{k+1}$ only contains monomials with total degree $2d_{k}\leq |\alpha|\leq d_{k+1}.$

We choose the sequence of constants $\{c_k>0\}$ at every step in such a way that
\be \|Q_{k+1}\|_{C^1(D_1, \R)}\leq \textrm{min}\{\delta_1, \ldots, \delta_k\}2^{-(k+1)}\ee
and that the power series $\sum_{k\geq 1}Q_k$ converges on the disk $D_2.$

Let now $\rho:\R^{n-1}\to [0, \infty)$ be a definable, smooth, cut-off function such that $\rho|_{D_1}\equiv 1$ and $\rho|_{D_2^c}\equiv 0$ and define the function $g:D^2\to \R$ by:
\be g(x)=\left(\sum_{k\geq2}Q_k(x)\right)\cdot \rho(x).\ee
We set $\hat{\Gamma}=\mathrm{graph}(g)\subset \R^n$ and extend this to a regular, compact definable manifold $\Gamma\subset \R^{n}.$ The set $D\subset \mathrm{graph}(g)\subset \Gamma$ will be the homeomorphic image of $D_1$ under the ``graph'' map $x\mapsto (x,g(x)).$

Let $P_1(x, y)=y$ and for every $k\geq 2$ define $P_k(x,y)=y-\sum_{j=2}^{k}Q_j(x).$ Observe that the degree of $P_k$ is $d_k$. For every $k\geq 1$ we consider now the (equivalent) systems of equations:
\be\{y -g(x)=0=
	P_k(x, y)\}\iff \bigg\{y -g(x)=0=
	Q_{k+1}(x)+\sum_{j\geq k+2}Q_j(x)=0\bigg\}
\ee
(the equivalence is obtained by eliminating $y$ from the second equation using the first one). The set of solutions to these systems in $D$ coincides with $Z(P_k)\cap D$.

Observe now that:
\be \bigg\|\sum_{j\geq k+2}Q_j\bigg\|_{C^1(D, \R)}\leq \sum_{j\geq k+2}\frac{\delta_k}{2^{j}}\leq \delta_k.\ee
In particular, since the equation $Q_{k+1}=0$ was regular on $D_1$, also the equation $Q_{k+1}+\sum_{j\geq k+2}Q_j=0$ is regular on $D_1$ and we have:
\be \bigg(Z\big(Q_{k+1}+\sum_{j\geq k+2}Q_j\big)\cap D_1, D_1\bigg)\sim (Z(Q_{k+1})\cap D_1, \sim D_1)\sim(Z_{d_k}, \R^{n-1}).\ee
As a consequence the system $\{y -g(x)=0=
	P_k(x, y)\}$ is regular on $D_1\times \R$ and under the graph map we have:
	\be (Z(P_k)\cap D, D)\sim (Z_{d_k}, \R^{n-1}).\ee
	Finally, let $p_k= {}^hP_k+R_k$ be a homogeneous polynomial (here ${}^hP_k$ denotes the homogenization) whose zero set is transverse to $\Gamma$ and with $\|R_k\|_{C^1(D, \R)}$ small enough such that 
	\be (Z(p_k)\cap D, D)\sim (Z_{d_k}, \R^{n-1}).\ee
	(The existence of such $R_k$ follows from the fact that  the set of homogeneous polynomials of a given degree whose zero set intersect $\Gamma$ transversely is dense).

\end{proof}

%% file: average-betti-numbers.tex
 \section{Estimates on the size of pathological examples: proof of Theorem \ref{thm:estimate}}
 Theorem \ref{thm:estimate} follows immediately from Proposition \ref{prop:bettimean} (proved below) after applying Markov's inequality.
In order to proceed we will need the following technical result. In the case $\Gamma$ is a real algebraic set this was proved by Gayet and Welschinger \cite{GaWe}. Our strategy of proof is also very similar, and it essentially uses the same ideas, just adapted to the non-algebraic setting. 
\begin{proposition}\label{thm:critical}Let $\Gamma\subset \mathbb{R}\emph{P}^n$ be a regular, compact hypersurface and $f:\Gamma\to \R$ be a Morse function. Let $p$ be a random Kostlan distributed polynomial on $\mathbb{R}\emph{P}^{n}$ of degree $d$. Then, denoting by $Q_{n-2}$ a \emph{GOE}$(n-2)$ matrix, we have:
\begin{equation}\label{eq:critE}\mathbb{E}\#\{\textrm{critical points of $f|_{\Gamma\cap Z(p)}$}\}=\frac{|\Gamma|}{\pi} \frac{d^{\frac{n-1}{2}}}{(2\pi)^{\frac{n-2}{2}}}\cdot  \mathbb{E}|\det Q_{n-2}|+O(d^{\frac{n-2}{2}}).\end{equation}
\end{proposition}

\begin{remark}Note that in the case $\dim \Gamma=1$ this can be obtained by a simple application of integral geometry.
\end{remark}
\begin{proof}
We will use the Kac-Rice formula for Riemannian manifolds. Since the involution $x\mapsto -x$ on the sphere with the round metric is an isometry, the quotient map $q:S^n\to \RP^n$ induces a Riemaniann metric on $\RP^n$ for which $q$ is a Riemannian submersion. In this way $\Gamma\subset \RP^n$ inherits a Riemannian metric as well. For every point $y\in\Gamma$ such that $d_yf\neq 0$ (since $\Gamma$ is compact and $f:\Gamma\to \R$ is Morse, there are only finitely many points where $d_yf$ vanishes) we consider an \emph{orthonormal frame field} $\{v_1, \ldots, v_{n-1}\}$ on a neighborhood $V\subset \Gamma$ of $y$ such that for all $x\in V$
\begin{equation} \ker d_x f=\textrm{span}\{v_2(x), \ldots, v_{n-1}(x)\}.\end{equation}

Let us take now an open set $V\subset \Gamma$  which is contained in the open set $\{x_0\neq 0\}\subset \RP^n$ (this is true after possibly shrinking $V$ and relabeling the homogeneous coordinates $[x_0, \ldots, x_n]$ in $\RP^n$). Let $\tilde{p}:\{x_0\neq 0\}\to \R$ be the random function defined by $\tilde p(x_0, x_1, \ldots, x_n)=p(1, x_1/x_0, \ldots, x_n/x_0)$ and denote by $\hat{p}$ its restriction to $V$: $\hat{p}=\tilde{p}|_{V}$ (thus $\hat{p}$ is a random function on the riemannian manifold $V\subset \Gamma\subset \RP^n$). Define the random map $F:V\to \R^{n-1}$ by:
\begin{equation} F(x)=(\hat{p}(x), d_x\hat p v_2(x), \ldots, d_x\hat pv_{n-1}(x)).\end{equation}
If the gradient of $\hat p$ does not vanish on $ \{p=0\}\cap V$ (this happens with probability one), then $\{p=0\}\cap V=\{\hat{p}=0\}$ is a smooth submanifold of $\Gamma$.  We claim that, with probability one, the number of critical points of $f|_{\{p=0\}\cap \Gamma}$ in $V$ equals the number of zeroes of $F$. In fact, with probability one, none of the critical points of $f$ lies on $\{p=0\}$ and in this case a point $x\in V$ is critical for $f|_{\{p=0\}\cap \Gamma}$ if and only if $\hat p(x)=0$ and the gradients of $\hat p$ and $f$ are collinear at $x$, i.e. $\hat{p}(x)=0$ and $\ker d_x \hat{p}=\ker d_x f$, which is equivalent to $F(x)=0$.

Let us denote by $\omega$ the volume density of $\Gamma$. Then the Kac-Rice formula for the random field $F$ on the Riemannian manifold $\Gamma\cap V$~\cite{adler2009random} tells that for any open set $W\subset V$
\begin{align} \mathbb{E}\#\{F=0\}\cap W&=\int_{W}\mathbb{E}\left\{|\det J(x)|\,\bigg|\, F(x)=0\right\} \rho_{F(x)}( 0) \omega(x) dx\\
&=\int_{W}\rho(x) \omega(x) dx.\end{align}
where the matrix $J(x)$ is the matrix of the derivatives at $x$ of the components of $F$ with respect to an orthonormal frame (in our case the chosen frame $v_1, \ldots, v_{n-1}$ and $\rho_{F(x)}(0)$ is the density at zero of the random vector $F(x).$

We use now the fact that the Kostlan polynomial $p$ is invariant by an orthogonal change of variable in $\RP^n$, hence for every $x\in V$ for the evaluation of 
$$\rho(x)=\mathbb{E}\left\{|\det J(x)|\,\bigg|\, F(x)=0\right\} \rho_{F(x)}( 0) $$
we can assume $x=[1, 0, \ldots, 0]=\overline{x}$. For simplicity let us also denote by $t_1, \ldots, t_n:\{x_0\neq 0\}\to \R$ the functions $t_i=x_i/x_0$.
 Then, since the stabilizer $O(n)$ of $\overline{x}$ acts transitively on the set of frames at $\overline x$, we can also assume that $\{v_1(\overline x), \ldots, v_{n-1}( \overline x)\}=\{\partial_1( \overline x), \ldots, \partial_{n-1}( \overline x)\}$, where we have denoted by $\partial_i$ the vector field $\partial/\partial t_i$.

For the calculation of the value of $\rho(\overline x)$ we use local coordinates on $\Gamma\cap V$. Note that $(t_1, \ldots, t_{n-1})$ are coordinates on $\Gamma\cap V$ (this is because the tangent space of $\Gamma$ at $\overline x$ equals $\textrm{span}\{\partial_1(\overline x), \ldots, \partial_{n-1}(\overline x)\}$). 
We denote by $\psi^{-1}:\{x_0\neq0\}\to \R^n$ the coordinate chart on $\{x_0\neq 0\}\subset \RP^n$ given by $(t_1, \ldots, t_n)$. In this chart $\psi^{-1}(\Gamma\cap V)$, for a small enough $V$ containing $\overline x$, can be seen as the graph of a function $g:\R^{n-1}\to \R$. Since the tangent space of $\psi^{-1}(\Gamma\cap V)$ at zero equals $\textrm{span}\{\partial_1, \ldots, \partial _{n-1}\}$, the function $g$ vanishes at zero, together with its differential. In this way we get a map $\varphi:\R^{n-1}\to \RP^n$ parametrizing $V$ given by:
\be \varphi(t_1, \ldots, t_{n-1})=\psi(t_1, \ldots, t_{n-1}, g(t_1, \ldots, t_{n-1})).\ee

Observe now that the frame $\{v_1, \ldots, v_{n-1}\}$ coincides with $\{\partial_{1}, \ldots, \partial_{n-1}\}$ only at zero; neverthless, it is easy to verify that we could pick the frame $\{v_1, \ldots, v_{n-1}\}$ such that in these coordinates:
\be v_i(t)=(1+t_i)\partial_i+O(\|t\|^2)\quad i=1, \ldots, n-1.\ee
In particular, denoting by $ t=(t_1, \ldots, t_{n-1})$, we have:
\be (d\hat pv_i)(t)=(1+t_i)\partial_{i}\hat p(t, g(t))+(1+t_i)\partial_n\hat p(t, g(t))\partial_ig(t)+O(\|t\|^2).\ee
From this it is immediate to see that:
\begin{align} F(\overline{x})&=\left(\hat p(0), (d\hat pv_1)(0), \ldots, (d\hat pv_{n-1})(0)\right)\\
&=(p_{d, 0, \ldots, 0}, p_{d-1, 0,1 0, \ldots, 0}, \ldots, p_{d-1, 0, \ldots, 0,1, 0, \ldots 0})\end{align}
(in the multi-index of the $i$-th entry of this vector the $1$ is in position $i+1$). In particular:
$$\rho_{F( \overline x)}(0)=\frac{1}{(2\pi)^{\frac{n-1}{2}}d^{\frac{n-2}{2}}}$$
Let us evaluate now the matrix $J( \overline x)$. 
For the first row $r_1( \overline x)$ of $J( \overline x)$ we immediately obtain:
$$r_1(\overline x)=(d_0\hat{p}v_1(0), \ldots, d_0\hat{p}v_{n-1}(0))=(\xi_{d-1, 1, \ldots, 0}, \xi_{d-1, 0,1 0, \ldots, 0}, \ldots, \xi_{d-1, 0, \ldots, 0,1})$$
Note that, except for the first entry, $r_1(\overline x)$ coincides with $F(\overline x)$; we denote by
\be w=(\xi_{d-1, 0,1 0, \ldots, 0}, \ldots, \xi_{d-1, 0, \ldots, 0,1})\ee
(i.e. the vector consisting of the last entries of the first row $r_1(\overline x)$).

Let us now look at the $(n-2)\times (n-2)$ submatrix $\hat{J}(\overline x)$ of $ J(\overline x)$, obtained by removing the first row and the first column. Observe that $\hat{J}(\overline x)=B+\xi_{d-1, 1,0, \ldots, 0}M(\overline{x})$, where $B$ is the matrix:
\be\label{eq:GOE} B= \left(\begin{array}{cccc}2\xi_{d-2, 0, 2,0, \ldots,0 } & \xi_{d-2, 0, 1, 1, 0, \ldots, 0} & \cdots & \xi_{d-2, 0, 1, 0\ldots, 0, 1} \\\xi_{d-2, 0, 1, 1, 0, \ldots, 0} & 2\xi_{d-2, 0, 0, 2, 0, \ldots, 0} & \cdots & \xi_{d-2, 0, 0, 1, 0, \ldots, 0, 1} \\\vdots &  &  &  \\\xi_{d-2, 0, 1, 0\ldots, 0, 1} & \xi_{d-2, 0, 1, 0, \ldots,0, 1} & \cdots & 2\xi_{d-2, 0, \ldots, 0, 2}\end{array}\right),\ee
and $M(\overline x)=(\partial_i\partial_j g(0))$.
From \eqref{eq:GOE} it is immediate to see that the matrix $B$ is a random matrix distributed as:
\begin{align}B=\sqrt{d(d-1)}Q_{n-2}\end{align}
where $Q_{n-2}$ is a random GOE$(n-2)$ matrix.
Hence
\be J(\overline{x})=\left(\begin{array}{c|c}\xi_{d-1, 1, 0, \ldots, 0} & w \\\hline * & B+\xi_{d-1, 1,0, \ldots, 0}M(\overline x)\end{array}\right)
\ee From this it follows that:
\begin{align}\mathbb{E}\left\{|\det J(\overline x)|\,\bigg|\, F(\overline x)=0\right\}&=\mathbb{E}\left\{|\det J(\overline x)|\,\bigg|\, w=0\right\}\\
&=(d(d-1))^{\frac{n-2}{2}}\mathbb{E}\left\{|\xi_{d-1, 1, 0, \ldots, 0}|\cdot |\det\left( Q_{n-2} +\frac{M(\overline x)}{\sqrt{d-1}}\right)|\,\bigg|\, w=0\right\} \\
&=(d(d-1))^{\frac{n-2}{2}}\cdot d^{\frac{1}{2}}\sqrt{\frac{2}{\pi}}\mathbb{E}\left|\det\left( Q_{n-2} +\frac{M(\overline x)}{\sqrt{d-1}}\right)\right|=(*),
\end{align}
where in the last step we have used the fact that the random variables $w$, $\xi_{d-1, 1, 0, \ldots, 0}$and $Q_{n-2}$ and $\xi_{d-1, 1, 0, \ldots, 0}$ are independent. Note now that, by construction, the matrix $M(\overline x)$ depends continuously on $\overline x\in \Gamma$, because we have assumed that $\Gamma$ is of regularity class $C^k$ with $k\geq 2$, and since $\Gamma$ is compact:
\be (*)=(d(d-1))^{\frac{n-2}{2}}\cdot d^{\frac{1}{2}}\sqrt{\frac{2}{\pi}}\left(\mathbb{E}|\det( Q_{n-2})|+O(d^{-1/2})\right).\ee
Putting all this together we obtain:
\begin{align} \mathbb{E}\#\{F=0\}\cap W&=\int_{W}\mathbb{E}\left\{|\det J(x)|\,\bigg|\, F(x)=0\right\} \rho_{F(x)}( 0) \omega(x) dx\\
&=\int_{W}\frac{(d(d-1))^{\frac{n-2}{2}}}{(2\pi)^{\frac{n-1}{2}}d^{\frac{n-2}{2}}}\cdot d^{\frac{1}{2}}\sqrt{\frac{2}{\pi}}\left(\mathbb{E}|\det( Q_{n-2})|+O(d^{-1/2})\right)\omega(x) dx\\
&=\frac{|W|}{\pi} \frac{d^{\frac{n-1}{2}}}{(2\pi)^{\frac{n-2}{2}}}\cdot  \mathbb{E}|\det Q_{n-2}|+O(d^{\frac{n-2}{2}}).\end{align}
From this the conclusion follows.
\end{proof}
In particular, since $f|_{\Gamma\cap\{p=0\}}$ is Morse with probability one (using standard arguments from differential topology it is not difficult to show that the set of such polynomials for which $f|_{\Gamma\cap\{p=0\}}$ is Morse has full measure), applying Morse's inequalities, we get the following corollary. \begin{proposition}\label{prop:bettimean}
There exists a universal constant $c_{k,n}>0$ such that
\be \mathbb{E}b_k(\Gamma\cap Z(p))\leq |\Gamma| c_{k, n}d^{\frac{n-1}{2}}+O(d^{\frac{n-2}{2}})\ee
(the implied constants in the $O(d^{\frac{n-2}{2}})$ depends on $\Gamma$).
\end{proposition}